\documentclass{scrartcl}
\usepackage[top=30truemm,bottom=40truemm,left=35truemm,right=35truemm]{geometry}
\usepackage{amssymb}
\usepackage{amsmath}
\usepackage{amsthm}
\usepackage{mathrsfs} 
\usepackage{amsfonts}
\usepackage[dvipdfmx]{graphicx}
\usepackage{bm}
\usepackage{color}
\usepackage{url}
\definecolor{mygreen}{cmyk}{0.64,0.00,0.95,0.40}

\newcommand{\dual}[1]{\langle{#1}\rangle}

\title{Hybridized Discontinuous Galerkin Method for Elliptic Interface Problems}
\author{
Masaru \textsc{Miyashita}%
\thanks{Technology Research Center, Sumitomo Heavy Industries,
 Ltd., Natsushima 19, Yokosuka, Kanagawa 237-8555, Japan.
 \textit{E-mail}: \texttt{masaru.miyashita.z@shi-g.com}}
\and 
Norikazu \textsc{Saito}% 
\thanks{Graduate School of Mathematical Sciences, The University of Tokyo,
Komaba 3-8-1, Meguro-ku, Tokyo 153-8914, Japan.
\textit{E-mail}: \texttt{norikazu@ms.u-tokyo.ac.jp}}
}
\begin{document}

%%% THEOREM STYLES
\theoremstyle{plain}
\newtheorem{thm}{Theorem}
\newtheorem{prop}[thm]{Proposition}
\newtheorem{cor}[thm]{Corollary}
\newtheorem{lemma}[thm]{Lemma}
\theoremstyle{remark}
\newtheorem{defi}[thm]{Definition}
\newtheorem{remark}[thm]{Remark}
\newtheorem{assum}[thm]{Assumption}
\newtheorem{ex}[thm]{Example}

%%% (1.1), (2.3), ....
%\makeatletter
%\renewcommand{\theequation}{%
%\arabic{section}.\arabic{equation}}
%\@addtoreset{equation}{section}
%\makeatother

%%
%%%
\maketitle

\begin{abstract}
New hybridized discontinuous Galerkin (HDG) methods for the interface
 problem for elliptic equations are proposed. Unknown functions of our
 schemes are $u_h$ in elements and $\hat{u}_h$ on inter-element
 edges. That is, we formulate our schemes without introducing the flux
 variable. Our schemes naturally satisfy the Galerkin orthogonality.  
The solution $u$ of the interface problem under consideration may not have a
 sufficient regularity, say $u|_{\Omega_1}\in H^2(\Omega_1)$ and   
$u|_{\Omega_2}\in H^2(\Omega_2)$, where $\Omega_1$ and $\Omega_2$ are
 subdomains of the whole domain $\Omega$ and
 $\Gamma=\partial\Omega_1\cap\partial\Omega_2$ implies the interface. We study the convergence, assuming $u|_{\Omega_1}\in H^{1+s}(\Omega_1)$ and   
$u|_{\Omega_2}\in H^{1+s}(\Omega_2)$ for some $s\in (1/2,1]$, where
 $H^{1+s}$ denotes the fractional order Sobolev space. Consequently, we
 succeed in deriving optimal order error estimates in an HDG norm and
 the $L^2$ norm. Numerical examples to validate our results are
 also presented. 
\end{abstract}

%%
%% Key words
%%
%\bigskip

\bigskip


\section{Introduction}
\label{sec:1}

Let $\Omega$ be a bounded domain in $\mathbb{R}^d$, $d=2,3$,
with the boundary $\partial\Omega$. We suppose
that $\Omega$ is divided into two disjoint subdomains $\Omega_1$ and
$\Omega_2$. Then, $\Gamma=\partial\Omega_1\cap\partial\Omega_2$ implies the
\emph{interface}. See Fig. \ref{f:1} for example.
%Below, we explicitly study 
%the case (I); the modification to the case (II) is straightforward. 

\begin{figure}[htb]
     \begin{minipage}{.49\textwidth}
	 \begin{center}
	  \includegraphics[width=.8\textwidth]{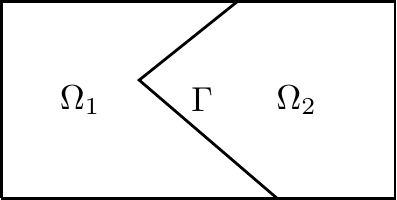}\\
	  Case (I) $\partial\Omega\cap\Gamma\ne\emptyset$.
	 \end{center}
    \end{minipage}
   \begin{minipage}{.49\textwidth}
	\begin{center}
	 \includegraphics[width=.8\textwidth]{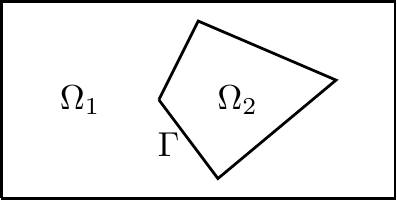}\\
	 Case (II) $\partial\Omega\cap\Gamma=\emptyset$.
	\end{center}
   \end{minipage}
   \caption{Examples of $\Omega_1$, $\Omega_2$ and $\Gamma$.}
\label{f:1}
\end{figure}
  
Suppose that we are given a matrix-valued function $A=A(x)$ of $\Omega\to \mathbb{R}^{d\times
d}$ such that: 
%\begin{subequations} 
%\label{eq:A}
  \begin{align*}
&\mbox{(smoothness)} && A|_{\Omega_i} \mbox{ is a $C^1$ function in }\Omega_i && (i=1,2);\\%\label{eq:K1}\\
&\mbox{(symmetry)} && \xi\cdot A(x)\eta=(A(x)\xi)\cdot \eta, && (\xi,\eta\in\mathbb{R}^d,x\in\Omega);\\%\label{eq:K2}\\
&\mbox{(elliptic condition)} && \lambda_{\min}|\xi|^2\le
   \xi\cdot A(x)\xi\le \lambda_{\max} |\xi|^2&& (\xi\in\mathbb{R}^d,x\in\Omega)%\label{eq:K3}
  \end{align*}
%\end{subequations}
with some positive constants $\lambda_{\min}$ and
$\lambda_{\max}$. Hereinafter, $|\cdot|=|\cdot|_{\mathbb{R}^d}$ denotes
the Euclidean norm in $\mathbb{R}^d$ and $\xi\cdot\eta$ the standard
scalar product in $\mathbb{R}^d$. 

We consider the following interface problem for second-order elliptic
equations for the function $u=u(x)$, $x\in\overline{\Omega}$, 
\begin{subequations} 
\label{eq:1}
  \begin{align}
 -\nabla \cdot A\nabla u&=f &&\mbox{in }\Omega\backslash\Gamma,\label{eq:1a}\\
   u&=0
   &&\mbox{on }\partial\Omega,\label{eq:1b}\\
   u|_{\Omega_1}-u|_{\Omega_2} &=g_D  && \mbox{on }\Gamma,\label{eq:1c}\\
   (A\nabla u)|_{\Omega_1}\cdot n_1+(A\nabla u)|_{\Omega_2}\cdot
   n_2&=g_N  && \mbox{on }\Gamma,\label{eq:1d}
  \end{align}
\end{subequations}
where $f,g_D,g_N$ are given functions, and $n_1,n_2$ are the unit normal
vectors to $\Gamma$ outgoing from $\Omega_1,\Omega_2$,  respectively. Moreover, $u|_{\Omega_1}$ stands for
the restriction of $u$ to $\Omega_1$ for example. 
We note that the gradient $\nabla u$ of the solution
may be discontinuous  across $\Gamma$, since $A$ may be discontinuous, 
even if $g_D=0$ and $g_N=0$. 

Elliptic interface problem of the form \eqref{eq:1} arises in many
fields of applications such as fluid dynamics and solid mechanics.
% In this study, we focus on the Radio Frequency (RF) plasma source by
% interior antenna.%
%
% The plasma source is expected to be high density and
% low metal contamination. However, the sputtering the antenna cover by
% high energy ion from sheath voltage still have been problematic.
%
For instance, the first author has proposed \eqref{eq:1} as a
 convenient model for computing
 sheath voltage wave
 form in the radio frequency plasma source within reasonable
 computational time (see \cite{miy16}). The model involves the interface where the
 electronic potential and flux have nontrivial gaps; see also \cite{gk97}. 
% has a discontinuous that electronic fluid equation in plasma connect to
% usual passion equation in antenna cover and chamber with semi
% analytical sheath interface.
%
% We estimate the sputtering distribution
% based on calculated sheath voltage waveform by this model, sputtering
% yield and ion energy distribution function (IEDF) model. The estimated
% sputtering distribution reproduce the tendency of experimental results.
%For example,
%the first author with his colleagues have dealt with \eqref{eq:1} to
%study reactive plasma deposition equipments . As a matter of fact,
%\eqref{eq:1} appeared as a part of the three dimensional hybrid PIC/MCC
%modelling.
The case $g_D=0$, which is sometimes referred to as elliptic problem
with discontinuous (diffusion) coefficients, is formulated as the
standard elliptic variational problem in $H^1(\Omega)$ and numerical
methods are studied by many authors; see \cite{cz98,bv00,pet02,cyz11}
for instance. On the other hand, the case
$g_D\ne 0$ has further difficulties and a lot of numerical methods have
been proposed (see \cite{bk96,mas12,mwwyz13} for example).

The present paper has dual purpose. 
The first one is to propose new schemes for solving
\eqref{eq:1} based on
the \emph{hybridized discontinuous Galerkin} (HDG) method. The HDG method is a
class of the discontinuous Galerkin (DG) method that is proposed by
Cockburn et al. (see \cite{cgl09}; see also \cite{ka72a,ka72b,ok10} for
other pioneering works). In the HDG method, we
introduce a new unknown function $\hat{u}_h$ on inter-element edges in
addition to the usual unknown function $u_h$ in elements. We can eliminate
$u_h$ from the resulting linear system and obtain the system only for $\hat{u}_h$; consequently, the size of the system
becomes smaller than that of the DG method. In this paper, we present
another advantage of the HDG method. That is, elliptic interface problem
\eqref{eq:1} is readily discretized by the HDG method and the
resulting schemes \eqref{eq:10} and \eqref{eq:100} described below naturally satisfy the consistency (see Lemma
\ref{la:3}) together with the Galerkin orthogonality (see
\eqref{eq:28}). It should be kept in mind that
Huynh et al. \cite{hnpk13} proposed an HDG scheme for \eqref{eq:1}. They
introduced further unknown function $q=A\nabla u$ and rewrote
\eqref{eq:1} into the system for $(u,q,\hat{u})$ based on the idea of
\cite{cgl09}, while our unknowns are only $(u,\hat{u})$ by following the
idea of \cite{ok10,oik10}. Herein, $\hat{u}$ denotes the trace of $u$
into inter-element edges. Moreover, results of numerical experiments
were well discussed and no theoretical consideration was undertaken in \cite{hnpk13}. 

The second purpose of this paper is to establish error estimates for the
HDG method when a sufficient regularity of solution, say $u\in
H^2(\Omega)$, could not be assumed. Actually, if $g_D\ne 0$,
the solution cannot be continuous across $\Gamma$. Moreover, we do not
always have partial regularities $u|_{\Omega_i}\in H^2(\Omega_i)$,
$i=1,2$. As a matter of fact, if
$\partial\Omega\cap\Gamma\ne\emptyset$, then we know that $u|_{\Omega_i}$ may
not belong to $H^2(\Omega_i)$, even when $\Gamma$ and $\partial\Omega$
are sufficiently smooth; see Remark \ref{rem:reg1}. 
To surmount of this obstacle, we employ the fractional order Sobolev
space $H^{1+s}(\Omega_i)$, $s\in (1/2,1]$, $i=1,2$, and are going to
attempt to derive an error estimate in an HDG norm $\|\cdot\|_{1+s,h}$
defined in terms of the $H^{1+s}(\Omega_i)$-seminorms (see \eqref{eq:31}). One of our final error estimate reads
(see Theorem \ref{th:2})
 \begin{equation*}
   \|\bm{u}-\bm{u}_h\|_{1+s,h}\le
Ch^s\left(\|u\|_{H^{1+s}(\Omega_1)}+\|u\|_{H^{1+s}(\Omega_2)}\right),
    %\label{eq:40}
 \end{equation*}
where $\bm{u}=(u,\hat{u})$ and $\bm{u}_h=(u_h,\hat{u}_h)$. 
Moreover, we also derive (see
Theorem \ref{th:3})
  \begin{equation*}
   \|u-u_h\|_{L^2(\Omega)}\le
Ch^{2s}\left(\|u\|_{H^{1+s}(\Omega_1)}+\|u\|_{H^{1+s}(\Omega_2)}\right),
    %\label{eq:40a}
  \end{equation*}
following the Aubin--Nitsche duality argument. To derive those
inequalities, we improve the standard boundness inequality for the
bilinear form (see Lemma \ref{la:7a}) and inverse inequality (see Lemma \ref{la:11a}) to fit our purpose. We note that those results are
actually optimal order estimates, since we assume only  
$u|_{\Omega_1}\in H^{1+s}(\Omega_1)$ and   
$u|_{\Omega_2}\in H^{1+s}(\Omega_2)$. 

In this paper, we concentrate our consideration on the case where
$\Omega_1$ and $\Omega_2$ are polyhedral domains in order to avoid unessential
complications about approximation of smooth surfaces/curves. The case of
a smooth $\Gamma$ is of great interest; we postpone it for future study. On the
other hand, we only consider the case $\partial\Omega\cap
\Gamma\ne\emptyset$, since the modification to the case
$\partial\Omega\cap\Gamma=\emptyset$ is readily and straightforward. 

\medskip

This paper is composed of five sections with an appendix. In Section \ref{sec:2}, we recall
the variational formulation of \eqref{eq:1} and state our HDG
schemes. The consistency is also proved there. The well-posedness of the
schemes is verified in Section \ref{sec:3}. Section \ref{sec:4} is
devoted to error analysis using the fractional order Sobolev
space. Finally, we conclude this paper by reporting numerical
examples to confirm our error estimates in Section \ref{sec:9}. In the
appendix, we state the proof of a modification of inverse inequality
(Lemma \ref{la:11a}).   
  
%\begin{subequations} 
%\begin{equation}
%[w]_\Gamma =w|_{\Omega_1}-w|_{\Omega_2} \quad \mbox{on }\Gamma 
%\label{eq:4}
%\end{equation}
%for a (suitably smooth) function $w$ defined in $\Omega$. 
%  \begin{align}
%[[A\nabla u]]_\Gamma &=(A\nabla u)\cdot n|_{\Omega_1}-u|_{\Omega_2} \quad \mbox{on }\Gamma. \label{eq:4b}
%  \end{align}
 %\end{subequations}

\section{Variational formulation and HDG schemes}
\label{sec:2}

For the geometry of $\Omega\subset\mathbb{R}^d$, $d=2,3$, we assume the following: 
\begin{equation}
 \tag{H1}
\mbox{$\Omega,~\Omega_1,\Omega_2$ are all polyhedral domains and
$\partial\Omega\cap \Gamma\ne\emptyset$}.  
\end{equation}
%Let $\Omega$ be a bounded polygonal domain in $\mathbb{R}^d$, $d=2,3$,
%with the boundary $\partial\Omega$. We suppose
%that $\Omega$ is divided into two disjoint subdomains $\Omega_1$ and
%$\Omega_2$ that are assumed to be polygonal domains. 
%Then, $\Gamma=\partial\Omega_1\cap\partial\Omega_2$ implies the
%\emph{interface}.
That is, we consider only Case (I) in Fig. \ref{f:1}. 

To state a variational formulation, we need several function spaces.
Namely, we use 
$L^2(\Omega)$, $H^m(\Omega)$, $m$ being a positive integer,
$H^1_0(\Omega)$, $L^2(\Gamma)$, 
$H^{1/2}(\Gamma)$, $H^{3/2}_0(\Gamma)$ and so on.
We follow the notation of \cite{lm72} for those Lebesgue and Sobolev
spaces and their norms. 
The standard seminorm of $H^m(\Omega)$ is denoted by
$|v|_{H^m(\Omega)}$.  
Supposing that $S$ is a part of $\partial\Omega$ or $\Gamma$, we let $\gamma(\Omega,{S})$ be the trace
operator from $H^1(\Omega)$ into $L^2(S)$. Set
\[
 H_{\Gamma}^1(\Omega_i)=\{v\in H^1(\Omega_i)\mid
\gamma(\Omega_i,\partial\Omega\cap\partial\Omega_i)v=0\},\quad i=1,2.
\]
Further set $\gamma_i=\gamma(\Omega_i,\Gamma)$, $i=1,2$. 
%We recall the following.
We introduce
\begin{equation*}
% \label{eq:V}
V=\{v\in L^2(\Omega)\mid v|_{\Omega_i}\in H^1_{\Gamma}(\Omega_i),\
i=1,2\}
\end{equation*}
and write $v_i=v|_{\Omega_i}$, $i=1,2$, for $v\in V$. 

Variational formulation of \eqref{eq:1} is given as follows: Find
$u\in V$ such that
\begin{subequations} 
\label{eq:2}
   \begin{gather}
    \gamma_1u_1 - \gamma_2u_2=g_D\quad \mbox{on }\Gamma, \label{eq:2a}\\
a(u,v)=\int_\Omega fv~dx+\int_\Gamma g_Nv~dS\qquad (\forall v\in H^1_0(\Omega)),\label{eq:2b}
   \end{gather}
where 
\begin{equation}
\label{eq:forma}
a(u,v)=\int_{\Omega_1}A\nabla u_1\cdot \nabla v_1~dx+
\int_{\Omega_2}A\nabla u_2\cdot \nabla v_2~dx.
\end{equation}
\end{subequations}

To state the well-posedness of Problem \eqref{eq:2}, we have to recall the so-called {Lions-Magenes space} (see \cite[\S 1.11.5]{lm72}) 
\[
 {H}^{1/2}_{00}(\Gamma)=\{\mu\in H^{1/2}(\Gamma)\mid
 \varrho^{-1/2}\mu\in L^2(\Gamma)\}
\]
which is a Hilbert space equipped with the norm $\|\mu\|_{{H}^{1/2}_{00}(\Gamma)}^2=\|\mu\|_{H^{1/2}(\Gamma)}^2+\|\varrho^{-1/2}\mu\|_{L^2(\Gamma)}^2$. 
Herein, $\varrho\in C^\infty(\overline{\Gamma})$ denotes any positive function satisfying
 $\varrho|_{\partial\Gamma}=0$ and, for $x_0\in\partial\Gamma$,  
$\lim_{x\to x_0}\varrho(x)/{\operatorname{dist~}(x,\partial\Gamma)}=\varrho_0>0$
with some $\varrho_0>0$. 
In particular, $H^{1/2}_{00}(\Gamma)$ is {strictly} included in
$H^{1/2}(\Gamma)$. 
The following result {follows directly} from \cite[Theorem 2.5]{gri76}
and \cite[Theorem 1.5.2.3]{gri85}. (A partial result is also reported in
\cite[Theorems 1.1 and 5.1]{sf00}.)

\begin{lemma}
The trace operator $v \mapsto \mu=\gamma_1 v$ is a linear and
 continuous operator of $H^1_\Gamma(\Omega_1)\to H^{1/2}_{00}(\Gamma)$. Conversely,  
there exists a linear and continuous operator $\mathcal{E}_1$ of $H^{1/2}_{00}(\Gamma)\to H^1_{\Gamma}(\Omega_1)$,
 which is called a lifting operator, such that $\gamma_1(\mathcal{E}_1\mu)=\mu$ for all $\mu\in H^{1/2}_{00}(\Gamma)$. The same propositions
 remain true if $\gamma_1$ and $\Omega_1$ are replaced by $\gamma_2$ and
 $\Omega_2$, respectively. 
\label{la:extension} 
\end{lemma}

Suppose that 
\begin{equation}
\tag{H2}
f\in L^2(\Omega),\quad g_D\in H^{1/2}_{00}(\Gamma)\quad \mbox{and}\quad g_N\in
L^2(\Gamma).
\end{equation}
In view of Lemma \ref{la:extension}, there is $\tilde{g}_D\in V$ such
that $\gamma_1\tilde{g}_D=\gamma_2\tilde{g}_D=g_D$ on $\Gamma$ and
$\|\tilde{g}_D\|_{H^1(\Omega)}\le C\|g_D\|_{H^{1/2}_{00}(\Gamma)}$.  

Hereinafter, the symbol $C$ denotes various generic positive constants depending on
$\Omega$. In particular, it is {independent of }the discretization
parameter $h$ introduced below. 
If it is necessary to specify the dependence on other parameters, 
say $\mu_1,\mu_2,\ldots$, then we write them as $C(\mu_1,\mu_2,\ldots)$. 

Therefore, we can apply the Lax--Milgram theory to conclude that the problem \eqref{eq:2} admits a unique
solution $u\in V$ satisfying  
\begin{equation*}
 \|u_1\|_{H^1(\Omega_1)}
  + \|u_2\|_{H^1(\Omega_2)}
  \le C(\|f\|_{L^2(\Omega)}+\|g_D\|_{H^{1/2}_{00}(\Gamma)}+
  \|g_N\|_{L^2(\Gamma)}),
 %\label{eq:2e}
\end{equation*}
where $C=C(A)$.

Next we review the regularity property of the solution $u$. Suppose further that
\begin{equation*}
%\tag{H2}
g_D\in H^{3/2}_{0}(\Gamma)\quad \mbox{and}\quad g_N\in
H^{1/2}(\Gamma).
\end{equation*}
However, in general, we do not expect that $u_1\in H^2(\Omega_1)$ and
$u_2\in H^2(\Omega_2)$, because of the presence of intersection points
$\Gamma\cap \partial\Omega$. (Even if we consider the case
$\Gamma\cap\partial\Omega=\emptyset$, we may have
$u_1\not\in H^2(\Omega_1)$ and $u_2\not\in H^2(\Omega_2)$.)
To state regularity properties of $u_1$ and $u_2$, it is useful to
introduce fractional order Sobolev spaces. We set
%use the semi-norms of fractional order Sobolev spaces:
%\begin{subequations}
% \label{eq:33}
%\begin{align}
%|v|_{H^\theta(\omega)}^2&=\int \hspace{-2mm} \int_{\omega\times\omega}\frac{~|v(x)-v(y)|^2}{|x-y|^{d+2\theta}}~dxdy,
% \label{eq:33a}\\
\begin{subequations} 
\label{eq:frac}
\begin{equation}
  |v|_{H^{1+\theta}(\omega)}^2=\sum_{i=1}^d\int \hspace{-2mm}
 \int_{\omega\times\omega}\frac{~|\partial_i v(x)-\partial_i v(y)|^2}{|x-y|^{d+2\theta}}~dxdy,
 \label{eq:33b}
\end{equation}
%\end{align}
%\end{subequations}
where $\omega\subset\mathbb{R}^d$, $\theta\in (0,1)$, and
$\partial_i=\partial/\partial x_i$. Then, fractional order
Sobolev spaces $H^{1+\theta}(\Omega_i)$, $i=1,2$, are defined as 
\begin{equation}
 \label{eq:330}
  H^{1+\theta}(\Omega_i)=\{v\in H^1(\Omega_i)\mid \|v\|_{H^{1+\theta}(\Omega_i)}^2=\|v\|_{H^1(\Omega_i)}^2+|v|_{H^{1+\theta}(\Omega_i)}^2<\infty\}.
\end{equation}
\end{subequations}

We assume that 
\begin{equation}
\tag{H2$'$}
g_D\in H_0^{s+1/2}(\Gamma)\quad \mbox{and}\quad g_N\in
H^{s-1/2}(\Gamma)
\end{equation}
and that the solution $u\in V$ of \eqref{eq:2} has the following
regularity property, %Instead, it is reasonable to \emph{assume}
\begin{equation}
 \label{eq:reg}
% \tag{H2}
 \left\{
 \begin{array}{l}
 u_1\in H^{1+s}(\Omega_1),\quad 
 u_2\in H^{1+s}(\Omega_2)\quad \mbox{and}\\[1mm]
N_s(u)  \le C(\|f\|_{L^2(\Omega)}+\|g_D\|_{H^{s+1/2}_{0}(\Gamma)}+
  \|g_N\|_{H^{s-1/2}(\Gamma)})
 \end{array}
 \right.
\end{equation}
for some $s\in (1/2,1]$, where $N_s(u)=\|u_1\|_{H^{1+s}(\Omega_1)}
  + \|u_2\|_{H^{1+s}(\Omega_2)}$ and $C=C(A)$.

\begin{remark}
 \label{rem:reg1}
We can find no explicit reference to \eqref{eq:reg}. 
Nevertheless, we consider the problem under \eqref{eq:reg} on the
 analogy of Poisson interface problem. As an illustration, we consider the case $d=2$. 
Suppose that $x_0$ is an intersection point of $\partial\Omega$
 and $\overline{\Gamma}$. We then set $U=\mathscr{O}\cap \Omega$ and
 $U_i=U\cap\Omega_i$, $i=1,2$, where $\mathscr{O}$ is a neighbourhood
 of $x_0$. 
Assume that $U$ contains no corners of $\partial\Omega\cup\Gamma$ and no other
 intersection points except for $x_0$. Consider the unique solution $w\in H^1_0(\Omega)$ of 
\begin{equation*}
\kappa_1\int_{\Omega_1} \nabla w\cdot \nabla v~dx+\kappa_2\int_{\Omega_2} \nabla w\cdot \nabla v~dx=\int_\Omega fv~dx\qquad (\forall v\in H^1_0(\Omega)),
%\label{eq:Poisson}
\end{equation*}
where $f\in L^2(\Omega)$ and $\kappa_1,\kappa_2$ are positive constants with
 $\kappa_1\ne \kappa_2$. Then, we have (see \cite[Theorem 6.2]{pet01})
\begin{equation*}
w|_{\Omega_i}\in H^{1+\beta}(U_i),\quad i=1,2,\quad
\beta=\min\left\{1,\frac{\pi}{2\theta}\right\}\in (1/2,1],
%\label{eq:5.5}
\end{equation*}
where $\theta$ denotes the maximum interior angle of $\partial\Omega_1$
 and $\partial\Omega_2$ at $x_0$. 
\end{remark}

We proceed to the presentation of our HDG schemes. 
We introduce a family of \emph{quasi-uniform} triangulations
$\{\mathcal{T}_h\}_h$ of $\Omega$. That is, $\{\mathcal{T}_h\}_h$ is a
family of \emph{shape-regular} triangulations that satisfies the
\emph{inverse assumptions} (see \cite[(4.4.15)]{bs08}). 
Hereinafter, we set $h=\max\{h_K\mid K\in \mathcal{T}_h\}$, where $h_K$
denotes the diameter of $K$. 
Let $\mathcal{E}_h=\{e\subset\partial K \mid K\in \mathcal{T}_h\}$ be
the set of all faces $(d=3)$/edges $(d=2)$ of elements, and set
$S_h=\cup_{K\in\mathcal{T}_h}\partial K=\cup_{e\in \mathcal{E}_h}e$. We
assume that there is a positive constant $\nu_1$ which is independent of $h$ such that 
\begin{equation}
 \tag{H3}
\max\left\{  \frac{h_e}{\rho_K},\ 
  \frac{h_K}{h_e}\right\}\le \nu_1
  \qquad (\forall e\subset\partial K,\ \forall
K\in\forall\mathcal{T}_h\in \{\mathcal{T}_h\}_h),
\end{equation}
where $h_e$ denotes the diameter of $e$ and $\rho_K$ the diameter of the
inscribed ball of $K$.

We use the following function spaces:
\begin{align*}
 H^{1+s}(\mathcal{T}_h)&=\{v\in L^2(\Omega)\mid v|_K\in H^{1+s}(K),\
 K\in\mathcal{T}_h\}; \\
 L^{2}_{\partial\Omega}(S_h)&=\{\hat{v}\in L^2(S_h)\mid \hat{v}|_e=0,\ e\in\mathcal{E}_{h,\partial\Omega}\};\\
 H^{1/2}_{\partial\Omega}(S_h)&=\{\hat{v}\in H^{1/2}(S_h)\mid \hat{v}|_e=0,\
 e\in\mathcal{E}_{h,\partial\Omega}\};\\
 \bm{V}^{1+s}(h)&=H^{1+s}(\mathcal{T}_h)\times H^{1/2}_{\partial\Omega}(S_h)
\end{align*}
for $s\in (1/2,1]$. 

Further, we assume that  
\begin{equation}
 \tag{H4}
\mbox{there exists a subset $\mathcal{E}_{h,\Gamma}$ of $\mathcal{E}_{h}$
  such that }\Gamma=\bigcup_{e\in \mathcal{E}_{h,\Gamma}}e,
\end{equation}
as shown for illustration in Fig. \ref{f:2}.  

\begin{figure}[tb]
	 \begin{center}
	  \includegraphics[width=.4\textwidth]{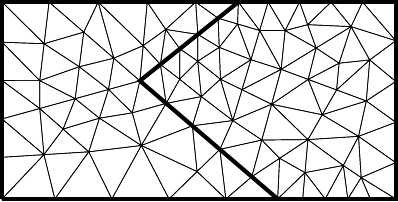}
	 \end{center}
   \caption{Triangulation satisfying (H4).}
\label{f:2}
\end{figure}

  We then set $\mathcal{E}_{h,\partial\Omega}=\{e\in\mathcal{E}_h\mid
e\subset\partial\Omega\}$ and $\mathcal{E}_{h,0}=\mathcal{E}_h\backslash
(\mathcal{E}_{h,\Gamma}\cup\mathcal{E}_{h,\partial\Omega})$. Assumption (H4) implies that $\mathcal{T}_{h,i}=\{K\in
\mathcal{T}_h\mid K\subset \overline{\Omega_i}\}$ is a triangulation of
$\Omega_i$ for $i=1,2$ and we can write   
\begin{equation}
 \label{eq:au0}
a(u,v)=\sum_{K\in\mathcal{T}_h}\int_K A\nabla u\cdot
 \nabla v~dx.
\end{equation}

Throughout this paper, we always assume that (H1), (H2), (H2$'$), (H3) and (H4) are
satisfied. 

For derivation of our HDG schemes, we examine a local conservation
property of the flux of the solution $u$.  
Let $K\in\mathcal{T}_h$.
%Since the solution $u$ of \eqref{eq:1} does not
%have enough regularity to define the trace $A\nabla u$ on $\partial K$
%as mentioned above,
%we interpret it as a functional.
Recall that, if $u$ is suitably regular, we have by \eqref{eq:1a} and Gauss--Green's formula
\[
 \int_{\Omega}(A\nabla u\cdot n_K) w~dS =\int_K A\nabla u\cdot
 \nabla w~dx-\int_K fw~dx
\]
for any $w\in H^{1}(K)$, where $n_K$ denotes the outer normal vector to
$\partial K$.
As mentioned above, the left-hand side of this identity is well-defined,
since \eqref{eq:reg} is assumed for some $s\in (1/2,1]$. However, we derive local
conservation properties (Lemmas \ref{la:1} and \ref{la:2} below) without
using the further regularity property \eqref{eq:reg}. That is, based on the
identity above, we introduce a
functional $\dual{A\nabla u\cdot n_K,\cdot}_{\partial K}$ on
$H^{1/2}(\partial K)$ by  
\begin{equation}
\label{eq:au}
 \dual{A\nabla u\cdot n_K,\phi}_{\partial K}=\int_K A\nabla u\cdot
 \nabla (Z\phi)~dx-\int_K f(Z\phi)~dx
\end{equation}
for any $\phi\in H^{1/2}(\partial K)$, where $Z\phi\in H^1(K)$ denotes a
suitable extension of $\phi$ such that $\|Z\phi\|_{H^1(K)}\le
C\|\phi\|_{H^{1/2}(\partial K)}$. Actually, the
definition of $\dual{A\nabla u\cdot n_K,\cdot}_{\partial K}$ above does
not depend on the way of extension of $\phi$. Below, for the solution
$u$ of \eqref{eq:1}, we simply write 
\begin{equation}
\label{eq:au2}
 \int_{\partial K}(A\nabla u\cdot n_K)\phi~dS=\int_K A\nabla u\cdot
 \nabla (Z\phi)~dx-\int_K f(Z\phi)~dx
\end{equation}
to express \eqref{eq:au}. 

The following lemmas are readily obtainable consequences of
\eqref{eq:au0} and \eqref{eq:au2}. 

\begin{lemma}
 \label{la:1}
 For the solution $u$ of \eqref{eq:2}, we have
 \begin{equation}
  \label{eq:5}
   \sum_{K\in\mathcal{T}_h}\int_{\partial K}(A\nabla u\cdot
   n_K)\hat{v}~dS= \int_\Gamma g_N\hat{v}~dS\qquad (\hat{v}\in H^{1/2}_{\partial\Omega}(S_h)).
 \end{equation}
\end{lemma}

\begin{lemma}
 \label{la:2}
 For the solution $u$ of \eqref{eq:2}, we have
 \begin{multline}
  \label{eq:6}
\sum_{K\in\mathcal{T}_h}\int_K A\nabla u\cdot\nabla v~dx
+   \sum_{K\in\mathcal{T}_h}\int_{\partial K}(A\nabla u\cdot
n_K)(\hat{v}-v)~dS\\
=
\sum_{K\in\mathcal{T}_h}\int_K fv~dx
+\int_\Gamma g_N\hat{v}~dS\qquad ((v,\hat{v})\in \bm{V}^{1}(h)).
 \end{multline}
\end{lemma}

We discretize the expression \eqref{eq:6} by the idea of HDG.
We use the following finite element spaces: 
 \begin{align*}
\bm{V}_h&=V_h\times \hat{V}_h;\\
V_h&=V_{h,k}=\{v\in H^{1}(\mathcal{T}_h)\mid v|_K\in P_k(K),\
  K\in\mathcal{T}_h\},\quad k\ge 1\mbox{: integer};\\
\hat{V}_h&=\hat{V}_{h,l}=\{\hat{v}\in L_{\partial\Omega}^2(S_h)\mid \hat{v}|_e\in P_l(e),\
  e\in\mathcal{E}_{h,0}\cup\mathcal{E}_{h,\Gamma}\},\quad l\ge 1\mbox{: integer},
 \end{align*}
where $P_k(K)$ denotes the set of all polynomials defined in $K$ with
degree $\le k$.

At this stage, we can state our scheme: Find $\bm{u}_h=(u_h,\hat{u}_h)\in
\bm{V}_h$ such that
\begin{subequations} 
\label{eq:10}
\begin{equation}
\label{eq:10a}
%\tag{\ref{eq:10}a}
 B_h(\bm{u}_h,\bm{v}_h)=L_h(\bm{v}_h)\qquad (\forall \bm{v}_h=(v_h,\hat{v}_h)\in\bm{V}_h),
\end{equation}
where
%\begin{align}
\begin{multline}
B_h(\bm{u}_h,\bm{v}_h) = \underbrace{\sum_{K\in\mathcal{T}_h}\int_K A\nabla
 u_h\cdot\nabla v_h~dx }_{=B_1} 
 %&\quad -
  \underbrace{-\sum_{K\in\mathcal{T}_h}\int_{\partial K}(A\nabla u_h\cdot
 n_K)({v}_h-\hat{v}_h)~dS}_{=B_2} \\
 \underbrace{-\sum_{K\in\mathcal{T}_h}\int_{\partial K}(A\nabla v_h\cdot
 n_K)({u}_h-\hat{u}_h)~dS}_{=B_3} %\label{eq:10d}\\ 
 \underbrace{+\sum_{K\in\mathcal{T}_h}\sum_{e\subset \partial
 K}\int_e\frac{\eta_e}{h_e} ({u}_h-\hat{u}_h)({v}_h-\hat{v}_h)~dS}_{=B_4}
\label{eq:10b}
\end{multline}
%\end{align}
and
\begin{multline}
 L_h(\bm{v}_h) = \underbrace{\sum_{K\in\mathcal{T}_h}\int_K fv_h~dx+\int_\Gamma
 g_N\hat{v}~dS}_{=L_1}%\label{eq:10f}\\
 %& \quad
 \underbrace{-\int_\Gamma g_D(A\nabla v_h\cdot n_1)~dS}_{=L_2} \\%\label{eq:10g}\\
\underbrace{+\sum_{e\in\mathcal{E}_{h,\Gamma}}\int_e\frac{\sigma_{K,e}}{2}\frac{\eta_e}{h_e}g_D
 ({v}_h-\hat{v}_h)~dS}_{=L_3}. \label{eq:10c}
\end{multline}
Therein, $\sigma_{K,e}$ is defined by 
\begin{equation}
 \sigma_{K,e}=
 \begin{cases}
 1 & (K\in \mathcal{T}_{h,1})\\
 -1 & (K\in \mathcal{T}_{h,2})
 \end{cases}
  \label{eq:10i}
\end{equation}
and $\eta_e$ denotes the penalty parameter such that
\begin{equation}
 0<\eta_{\min}=\inf_{\mathcal{T}_h\in\{\mathcal{T}_h\}_h}
  \min_{e\in\mathcal{E}_h}\eta_e,\qquad
\eta_{\max}=\sup_{\mathcal{T}_h\in\{\mathcal{T}_h\}_{h}}\max_{e\in\mathcal{E}_h}\eta_e<\infty.
  \label{eq:10j}
\end{equation}
\end{subequations}

The main advantage of the scheme \eqref{eq:10} is stated as the
following lemma.

\begin{lemma}[Consistency]
 \label{la:3}
Let $u\in V$ be the solution of \eqref{eq:2} and introduce $\hat{u}\in
 H^{1/2}_{\partial\Omega}(S_h)$ by
\[
 \hat{u}=
 \begin{cases}
\frac{1}{2}[\gamma(K_1,e)u+\gamma(K_2,e)u] &
  (e\in\mathcal{E}_{h,0}\cup\mathcal{E}_{h,\Gamma},\ e=\partial K_1\cap
  \partial K_2)\\
\gamma(K,e)u& (e\in \mathcal{E}_{h,\partial\Omega},\ e\subset\partial K).  
 \end{cases}
 \]
 Then, $\bm{u}=(u,\hat{u})\in \bm{V}^1(h)$ solves 
 \[
   B_h(\bm{u},\bm{v}_h)=L_h(\bm{v}_h)\qquad (\forall \bm{v}_h\in\bm{V}_h).
 \]
\end{lemma}

\begin{proof}
In view of Lemma \ref{la:2}, we know that 
 $B_1+B_2=L_1$.
 We show that
 $B_3=L_2$ and $B_4=L_3$.
 For $e\in
 \mathcal{E}_{h,0}\cap \mathcal{E}_{h,\partial\Omega}$, we
 have $u-\hat{u}=0$ on $e$, since $\hat{u}=\gamma(\Omega,\Gamma)u$ on $e$. Hence,
\begin{align*}
 B_3
 %&=\sum_{K\in\mathcal{T}_h}\int_{\partial K}(A\nabla v_h\cdot 
 % n_K)\cdot({u}_h-\hat{u}_h)~dS \\
 &=\sum_{e\in\mathcal{E}_{h,\Gamma}}
 \left[
 \int_e (A\nabla v_h\cdot n_{1})({u}_h-\hat{u}_h)~dS+
 \int_e (A\nabla v_h\cdot n_{2})({u}_h-\hat{u}_h)~dS \right]\\
 &=\sum_{e\in\mathcal{E}_{h,\Gamma}}
 \left[
 \int_e (A\nabla v_h\cdot n_{1})\frac{u_1-u_2}{2}~dS-
 \int_e (A\nabla v_h\cdot n_{1})\frac{u_2-u_1}{2}~dS \right]\\
 &=%\sum_{e\in\mathcal{E}_{h,\Gamma}}
 \int_\Gamma (A\nabla v_h\cdot n_{1}) (u_1-u_2)~dS,
\end{align*}
where $e=\partial K_1\cap\partial K_2$ with $K_1\in\mathcal{T}_{h,1}$
 and $K_2\in\mathcal{T}_{h,2}$. This, together with \eqref{eq:2a}, gives
 $B_3=L_2$. Using the same notion, we have
\begin{align*}
 %\eqref{eq:10e}
B_4
 &=\sum_{e\in\mathcal{E}_{h,\Gamma}}
 \left[
 \int_e\frac{\eta_e}{h_e} ({u}_1-\hat{u})({v}_{h,1}-\hat{v}_h)~dS+
 \int_e\frac{\eta_e}{h_e} ({u}_2-\hat{u})({v}_{h,2}-\hat{v}_h)~dS
 \right]\\
 &=\sum_{e\in\mathcal{E}_{h,\Gamma}}
 \left[
 \int_e\frac{\eta_e}{h_e} \frac{g_D}{2}({v}_{h,1}-\hat{v}_h)~dS
 -\int_e\frac{\eta_e}{h_e}\frac{g_D}{2}({v}_{h,2}-\hat{v}_h)~dS
 \right]=L_3,%\eqref{eq:10h}. 
\end{align*}
 which completes the proof. 
\end{proof}

% \begin{remark}
%  \label{rem:scheme2}
An alternative scheme is given as
\begin{subequations} %16:09$B$N<072(B
 \label{eq:100}
   \begin{equation}
 \label{eq:10z}
 B_h(\bm{u}_h,\bm{v}_h)=L_h'(\bm{v}_h)\qquad (\forall \bm{v}_h\in\bm{V}_h),
  \end{equation}
  where
%  considered by prelacing \eqref{eq:10h} by
 \begin{equation} 
L_h'(\bm{v}_h)=L_1+L_2+\sum_{e\in\mathcal{E}_{h,\Gamma}}\int_e\sigma_{K,e}'\frac{\eta_e}{h_e}
 g_D({v}_h-\hat{v}_h)~dS
 \label{eq:10hh}
 \end{equation}
and  
\begin{equation}
 \sigma_{K,e}'=
 \begin{cases}
 1 & (K\in \mathcal{T}_{h,1})\\
 0 & (K\in \mathcal{T}_{h,2}).
 \end{cases}
  \label{eq:10ii}
\end{equation}
\end{subequations}
 Lemma \ref{la:3} remains valid for \eqref{eq:100} with an obvious
  modification of the definition of $\hat{u}$. Therefore, all the
  following results also remain true for \eqref{eq:100}. Hence, we
  explicitly study only \eqref{eq:10} below. 
%\end{remark}

 \begin{remark}
  \label{rem:17}
  We restrict ourselves to simplicial triangulations; that is, we are
  assuming that each $K\in\mathcal{T}_h\in\{\mathcal{T}_h\}_h$ is a
  $d$-simplex. However, we are able to consider more general shape of
  elements. For example, for $d=2$, each $K$ could be an $m$-polygonal domain, where $m$ is an integer and can
differ with $K$. We assume that $m$ is bounded from above
independently of a family of triangulations and
  $\partial K$ does not intersect with itself. In particular, we can
  consider rectangular meshes as well. Moreover, $\bm{V}_h$ could be
  replaced by any finite dimensional subspace $\bm{V}_h'$ of
  $\bm{V}^1(h)$. See \cite{oik10,ok10} for the detail of modifications. 
 \end{remark}

\section{Well-posedness}
\label{sec:3}

In this section, we establish the well-posedness of the scheme \eqref{eq:10}.
First, we recall the following standard results; \eqref{eq:iv} is the
standard inverse inequality (see \cite[Lemma 4.5.3]{bs08}) and
\eqref{eq:tr} follows from the standard trace inequalities (see also
Appendix \ref{sec:ap}). 

\begin{lemma}
 \label{la:11}
 For $K\in \mathcal{T}_h$, we have following inequalities. 

\noindent (Inverse inequality)
       \begin{equation}
	\label{eq:iv}
	 |v_h|_{H^1(K)}\le C_{\mathrm{IV}}h_{K}^{-1}\|v_h\|_{L^2(K)}\qquad (v_h\in V_h).
       \end{equation}

\noindent (Trace inequalities) 
  \begin{subequations} \label{eq:tr}
   \begin{align}
	  \|v\|_{L^2(e)}^2 &\le
	  C_{0,\mathrm{T}}h_e^{-1}\left(\|v\|_{L^2(K)}^2+h_K^2|v|_{H^1(K)}^2\right)&&(v\in
    H^1(K)), \label{eq:tr1}\\
	  \|\nabla v\|_{L^2(e)}^2 &\le
    C_{1,\mathrm{T}}h_e^{-1}\left(\|v\|_{H^1(K)}^2+h_K^2|v|_{H^2(K)}^2\right)&& (v\in
    H^2(K)).\label{eq:tr2}
   \end{align}
  \end{subequations}
Those $C_{\mathrm{IV}}$, $C_{0,\mathrm{T}}$ and $C_{1,\mathrm{T}}$ are
 absolute positive constants. 
\end{lemma}

We use the following HDG norms:
\begin{subequations} %09:32$B$N<072(B
 \label{eq:21}
\begin{align}
  \|\bm{v}\|_{1,h}^2&=
  \sum_{K\in\mathcal{T}_h}|v|_{H^1(K)}^2+
  \sum_{K\in\mathcal{T}_h}\sum_{e\subset\partial K}
 {\frac{\eta_e}{h_e}}\left\|\hat{v}-v\right\|_{L^2(e)}^2;
 \label{eq:21a}\\
  \|\bm{v}\|_{2,h}^2 &=
  \sum_{K\in\mathcal{T}_h}|v|_{H^1(K)}^2+
  \sum_{K\in\mathcal{T}_h}h_K^2|v|_{H^2(K)}^2+
 \sum_{K\in\mathcal{T}_h}\sum_{e\subset\partial K}
 \frac{\eta_e}{h_e}
 \left\|\hat{v}-v\right\|_{L^2(e)}^2.
 \label{eq:21b}
%  |\bm{v}|_{1,h}^2+
%  |\bm{v}|_{j,h}^2\quad (\bm{v}\in H^1(\mathcal{T}_h)\times H^{1/2}_{\partial\Omega}(S_h)),
\end{align}
\end{subequations}
%for $\bm{v}=H^1(\mathcal{T}_h)\times H^{1/2}_{\partial\Omega}(S_h)$. 
%where
%\begin{align*}
% |\bm{v}|_{1,h}^2&=\sum_{K\in\mathcal{T}_h}|v|_{H^1(K)}^2,\\
% |\bm{v}|_{j,h}^2&=\sum_{K\in\mathcal{T}_h}\sum_{e\subset\partial K}
% \left|\sqrt{\frac{\eta_e}{h_e}}(\hat{v}-v)\right|_{L^2(e)}^2.
%\end{align*}
Moreover, set
%\begin{subequations} %16:55$B$N<072(B
% \label{eq:23}
\[
 \alpha =\max\left\{
\sup_{x\in{\Omega}_1}
  \sup_{{\xi}\in\mathbb{R}^d}
\frac{|A(x){\xi}|}{|{\xi}|},\ 
\sup_{x\in{\Omega}_2}
  \sup_{{\xi}\in\mathbb{R}^d}
\frac{|A(x){\xi}|}{|{\xi}|}
 \right\}.% \label{eq:23a}\\
\]
 %\eta_{\min}&=\min\{\eta_e\mid e\in\mathcal{E}_h \} .\label{eq:23b}
%\end{align}
%\end{subequations}

\begin{remark}
 \label{rem:20}
In view of \eqref{eq:iv}, two norms
 $\|\bm{v}\|_{1,h}$ and $\|\bm{v}\|_{2,h}$ are equivalent norms in the
 finite dimensional space $\bm{V}_h$. That is, there exists a positive
 constant $C_0$ that depends only on $C_{\mathrm{IV}}$ such that 
\begin{equation}
 \label{eq:eqv}
\|\bm{v}\|_{1,h}\le \|\bm{v}\|_{2,h}\le C_0 \|\bm{v}\|_{1,h}\qquad (\bm{v}_h\in\bm{V}_h).
\end{equation}
%In particular, $C_0$ is independent of $h$. 
%\cred{We should assume $\nabla v_h\in V_h^d$ for $v_h^in V_h$.}
\end{remark}

\begin{lemma}
 \label{la:7}
{(Boundness)} For any $\eta_{\min}>0$, there exists a positive constant $C_{\mathrm{b}}=C_{\mathrm{b}}(\alpha,\eta_{\min},d,C_{1,\mathrm{T}})$ such that
       \begin{equation}
	\label{eq:bdd}
	 B_h(\bm{w},\bm{v})\le C_{\mathrm{b}}\|\bm{w}\|_{2,h}\|\bm{v}\|_{2,h}\qquad
	 (\bm{w},\bm{v}\in \bm{V}^2(h)).
       \end{equation}
%       The constant $C_{\mathrm{b}}$ depends on
% and are independent of $h$. 

 \noindent {(Coercivity)} There exist positive constants $\eta^*=\eta^*(\alpha,\lambda_{\min},d,C_{1,\mathrm{T}},C_{\mathrm{IV}})$ and
	      $C_{\mathrm{c}}=C_{\mathrm{c}}(\lambda_{\min},C_{\mathrm{IV}})$ such that, if $\eta_{\min}\ge \eta^*$, we have 
       \begin{equation}
	\label{eq:coer}
	 B_h(\bm{v}_h,\bm{v}_h)\ge C_{\mathrm{c}}\|\bm{v}_h\|_{2,h}^2\qquad
	 (\bm{v}_h\in \bm{V}_h).
       \end{equation}
%\end{enumerate}
%The constant $C_{\mathrm{c}}$ depends only on $\lambda_{\min}$ and
% $C_{\mathrm{IV}}$, 
% while $\eta^*$ on
% $\alpha,\lambda_{\min},\eta_{\min},d,C_{1,\mathrm{T}},C_{\mathrm{IV}}$. In
% particular, those constants are both independent of $h$. 
\end{lemma}

Both inequalities are essentially well-known; however, we briefly state
 their proofs, since the contribution of parameters on $C_{\mathrm{c}}$
 and $C_{\mathrm{b}}$ should be clarified. Moreover, we shall state the
 extension of \eqref{eq:bdd} below (see Lemma \ref{la:7a}) so it is useful to
 recall the proof of \eqref{eq:bdd} at this stage.

 \begin{proof}[Proof of Lemma \ref{la:7}] 
{(Boundness)} Let $\bm{w}=(w,\hat{w}),\bm{v}=(v,\hat{v})\in
 \bm{V}^2(h)$.
For $e\subset \partial K$, $K\in\mathcal{T}_h$, we
 have by Schwarz' inequality 
\begin{multline*}
\left|
 \int_{e}(A\nabla w\cdot
 n_K)({v}-\hat{v})~dS\right|
% & \le \int_{e}|A\nabla w_h| \cdot | n_K|\cdot |{v}_h-\hat{v}_h|~dS\\
% & \le \alpha \int_{e}|\nabla w_h| \cdot |{v}_h-\hat{v}_h|~dS\\
 \le \alpha \left(\frac{h_e}{\eta_e}\right)^{1/2}\|\nabla
 w\|_{L^2(e)}\cdot \left(\frac{\eta_e}{h_e}\right)^{1/2}\|{v}-\hat{v}\|_{L^2(e)}.%\\
\end{multline*}
 % &\le \alpha
% C_{1,\mathrm{T}}\eta_e^{-1/2}\left(|w_h|_{H^1(K)}^2+h_K^2|w_h|_{H^2(K)}^2\right)^{1/2}\cdot
% \left(\frac{\eta_e}{h_e}\right)^{1/2}\|{v}_h-\hat{v}_h\|_{L^2(e)}.%\\
 %&\le \alpha
 %C_{1,\mathrm{T}}\eta_e^{-1/2}\left(|w_h|_{H^1(K)}^2+h_K^2|w_h|_{H^2(K)}^2+
%\frac{\eta_e}{h_e}\|{v}_h-\hat{v}_h\|_{L^2(e)}^2\right).
 %&\le \frac{\alpha}{2}
% C_{1,\mathrm{T}}\left(|w_h|_{H^1(K)}+h_K|w_h|_{H^2(K)}\right)\cdot h_e^{-1/2}\|{v}_h-\hat{v}_h\|_{L^2(e)}.
 %\end{align*}
 Hence, using Schwarz' inequality again,  
\begin{align*}
B_h(\bm{w},\bm{v})&\le \sum_{K\in \mathcal{T}}\alpha
 |w|_{H^1(K)}|v|_{H^1(K)}\\
 &+\sum_{K\in\mathcal{T}_h}\sum_{e\subset \partial K}
 \frac{\alpha}{\eta_{\min}^{1/2}}h_e^{1/2} \|\nabla w\|_{L^2(e)}\cdot
 \left(\frac{\eta_e}{h_e}\right)^{1/2}\|{v}-\hat{v}\|_{L^2(e)}\\
 &+\sum_{K\in\mathcal{T}_h}\sum_{e\subset \partial K}
 \frac{\alpha}{\eta_{\min}^{1/2}}h_e^{1/2} \|\nabla v\|_{L^2(e)}\cdot
 \left(\frac{\eta_e}{h_e}\right)^{1/2}\|{w}-\hat{w}\|_{L^2(e)}\\
 &+\sum_{K\in\mathcal{T}_h}\sum_{e\subset \partial K}
 {\frac{\eta_e}{h_e}}\left\|{w}-\hat{w}\right\|_{L^2(e)}\cdot 
 {\frac{\eta_e}{h_e}}\left\|{v}-\hat{v}\right\|_{L^2(e)}\\
& \le C\left[\sum_{K\in \mathcal{T}_h} |w|_{H^1(K)}^2
 +\sum_{e\subset \partial K}\left(h_e^{-1}\|\nabla
 w\|_{L^2(e)}^2
 +{\frac{\eta_e}{h_e}}\left\|{w}-\hat{w}\right\|_{L^2(e)}^2\right)\right]^{1/2}\\
& \cdot \left[\sum_{K\in \mathcal{T}_h} |v|_{H^1(K)}^2
 +\sum_{e\subset \partial K}\left(h_e^{-1}\|\nabla v\|_{L^2(e)}^2
 +
 {\frac{\eta_e}{h_e}}\left\|{v}-\hat{v}\right\|_{L^2(e)}^2\right)\right]^{1/2}.
 % |B_2| &\le \frac{\alpha C_{1,\mathrm{T}}}{\eta_e}
%\sum_{K\in\mathcal{T}_h}\left[
%  (d+1)\left(|w_h|_{H^1(K)}^2+h_K^2|w_h|_{H^2(K)}^2\right)+
% \sum_{e\in\partial
% K}\left\|\sqrt{\frac{\eta_e}{h_e}}({v}_h-\hat{v}_h)\right\|_{L^2(e)}^2\right]
% &\le \frac{\alpha C_{1,\mathrm{T}}}{\eta_e}(d+1)
\end{align*}
Therefore, using \eqref{eq:tr2}, we obtain \eqref{eq:bdd}. 
% \[
%  B_h(\bm{w},\bm{v})\le C\|\bm{w}\|_{2,h}\|\bm{v}\|_{2,h}.
% \]
%This, together with \eqref{eq:eqv}, gives \eqref{eq:bdd}. Therein, $C_3$
% denotes a positive constant defined by
% $\alpha,d,\eta_{\min}$, while $C_4$ is defibed by $C_3$
% and $C_{1,\mathrm{T}}$.
 
 \smallskip
 
 \noindent {(Coercivity)}  Let $\bm{v}_h=(v_h,\hat{v}_h)\in
 \bm{V}_h$. Then, 
%\begin{align*}
\begin{multline*}
B_h(\bm{v}_h,\bm{v}_h)
%&= \sum_{K\in\mathcal{T}_h}\int_K A\nabla v_h\cdot\nabla v_h~dx \\
%&\quad -\underbrace{2\sum_{K\in\mathcal{T}_h}\int_{\partial K}(A\nabla v_h\cdot
% n_K)\cdot({v}_h-\hat{v}_h)~dS}_{=J} \\
%&\quad +\sum_{K\in\mathcal{T}_h}\sum_{e\subset \partial K}\int_e\frac{\eta_e}{h_e}
% ({v}_h-\hat{v}_h)^2~dS\\
\ge  \lambda_{\min}\sum_{K\in\mathcal{T}_h}|v_h|_{H^1(K)}^2\\
 +\sum_{K\in\mathcal{T}_h}\sum_{e\subset \partial K}
{\frac{\eta_e}{h_e}}\left\|\hat{v}-v\right\|_{L^2(e)}^2-2\sum_{K\in\mathcal{T}_h}\int_{\partial K}(A\nabla v_h\cdot n_K)({v}_h-\hat{v}_h)~dS.%\\
\end{multline*} 
%&\ge  \min{\lambda_{\min},1}\lambda_{\min}\sum_{K\in\mathcal{T}_h}|v_h|_{H^1(K)}^2
% +\sum_{K\in\mathcal{T}_h}\sum_{e\subset \partial K}
%\left\|\sqrt{\frac{\eta_e}{h_e}}(\hat{v}-v)\right\|_{L^2(e)}^2-2J.%\\
%&\quad +\sum_{K\in\mathcal{T}_h}\sum_{e\subset \partial K}\int_e\frac{\eta_e}{h_e}
% ({v}_h-\hat{v}_h)^2~dS,
%\end{align*}
Letting $e\subset \partial K$, $K\in\mathcal{T}_h$, we
 have by \eqref{eq:tr2}, \eqref{eq:iv}, Schwarz' and Young's inequalities 
\begin{align*}
& \left|
 \int_{e}(A\nabla v_h\cdot
 n_K)({v}_h-\hat{v}_h)~dS\right|\\
% & \le \int_{e}|A\nabla w_h| \cdot | n_K|\cdot |{v}_h-\hat{v}_h|~dS\\
% & \le \alpha \int_{e}|\nabla w_h| \cdot |{v}_h-\hat{v}_h|~dS\\
 &\le \alpha \|\nabla v_h\|_{L^2(e)}\|{v}_h-\hat{v}_h\|_{L^2(e)}\\
 &\le \alpha
 C_{1,\mathrm{T}}h_e^{-1/2}\left(|v_h|_{H^1(K)}^2+h_K^2|v_h|_{H^2(K)}^2\right)^{1/2}\cdot
\|{v}_h-\hat{v}_h\|_{L^2(e)}\\
% &\le \alpha
% C_{1,\mathrm{T}}\left(1+C_{\mathrm{IV}}^2\right)^{1/2}|v_h|_{H^1(K)}\cdot
%{h_e}^{-1/2}\|{v}_h-\hat{v}_h\|_{L^2(e)}\\
 &\le 
 C_*(\delta\eta_e)^{-1/2}|v_h|_{H^1(K)}\cdot
\left(\frac{\eta_e\delta}{h_e}\right)^{1/2}\|{v}_h-\hat{v}_h\|_{L^2(e)}\\
 &\le 
 \frac{C_*^2}{\delta\eta_e}|v_h|^2_{H^1(K)}+
\delta{\frac{\eta_e}{h_e}}\left\|{v}_h-\hat{v}_h\right\|_{L^2(e)}^2,%\\
 %&\le \alpha
 %C_{1,\mathrm{T}}\eta_e^{-1/2}\left(|v_h|_{H^1(K)}^2+h_K^2|v_h|_{H^2(K)}^2\right)^{1/2}\cdot
 %\left(\frac{\eta_e}{h_e}\right)^{1/2}\|{v}_h-\hat{v}_h\|_{L^2(e)}\\
 %&\le \alpha
 %C_{1,\mathrm{T}}\eta_e^{-1/2}\left(|v_h|_{H^1(K)}^2+h_K^2|v_h|_{H^2(K)}^2+
 %\frac{\eta_e}{h_e}\|{v}_h-\hat{v}_h\|_{L^2(e)}^2\right).
 %&\le \frac{\alpha}{2}
% C_{1,\mathrm{T}}\left(|w_h|_{H^1(K)}+h_K|w_h|_{H^2(K)}\right)\cdot h_e^{-1/2}\|{v}_h-\hat{v}_h\|_{L^2(e)}.
\end{align*}
where $C_*=C_*(\alpha,d,C_{1,\mathrm{T}},C_{\mathrm{IV}})$ and $\delta$ is a positive constant specified later.  
 Using this, we deduce 
 \begin{multline*}
   B_h(\bm{v}_h,\bm{v}_h)\ge
 \left[\lambda_{\min}-2(d+1)\frac{C_*^2}{\delta\eta_{\min}}\right]
  \sum_{K\in\mathcal{T}_h}|v_h|_{H^1(K)}^2\\
  +(1-2\delta)\sum_{K\in\mathcal{T}_h}\sum_{e\subset \partial K}
{\frac{\eta_e}{h_e}}\left\|\hat{v}-v\right\|_{L^2(e)}^2.
 \end{multline*}
At this stage, choosing $\delta$ and $\eta_{\min}$ such that 
\begin{equation*}
% \label{eq:eta}
0<  \delta\le \frac14,\qquad \eta_{\min}\ge 4(d+1)\frac{C_*^2}{\lambda_{\min}\delta},
\end{equation*}
 we obtain
\[
 B_h(\bm{v}_h,\bm{v}_h)\ge \frac{1}{2}\min\{1,~\lambda_{\min}\}\|\bm{v}_h\|_{1,h}^2,
\]
which, together with \eqref{eq:eqv}, implies \eqref{eq:coer}. 
 \end{proof}

%\begin{lemma}
% \label{eq:L}
%There exists a positive constant $C_3$ such that
%\begin{equation}
%\label{eq:L}
%|L_h(\bm{v})|\le C\|\bm{v}\|_{2,h}\qquad
% (\bm{v}\in \bm{V}^2(h)).
%\end{equation}
%The constant $C_3$ depends only on 
% $\alpha,\lambda_{\min},\Omega$, $f$, $g_D$ and $g_N$;  $C_3$ is independent of $h$. 
%\end{lemma}

\section{Error analysis}
\label{sec:4}

This section is devoted to error analysis of our HDG scheme. 
%and
%\begin{equation}
% \label{eq:33b}
% I_{\theta}(v)=\int \hspace{-2mm} \int_{\omega\times\omega}\frac{~|v(x)-v(y)|^2%}{|x-y|^{d+2\theta}}~dxdy.
%\end{equation}
We use a new HDG norm:
\begin{equation} 
 \label{eq:31}
  \|\bm{v}\|_{1+s,h}^2 =
  \sum_{K\in\mathcal{T}_h}|v|_{H^1(K)}^2+
  \sum_{K\in\mathcal{T}_h}h_K^{2s}|v|_{H^{1+s}(K)}^2+
  \sum_{K\in\mathcal{T}_h}\sum_{e\subset\partial K}
 {\frac{\eta_e}{h_e}}\left\|\hat{v}-v\right\|_{L^2(e)}^2
%  |\bm{v}|_{1,h}^2+
%  |\bm{v}|_{j,h}^2\quad (\bm{v}\in H^1(\mathcal{T}_h)\times H^{1/2}_{\partial\Omega}(S_h)),
\end{equation}
for $s\in (1/2,1)$.

We have to improve Lemmas \ref{la:11} and \ref{la:7} for our purpose. 
First, the trace inequality for functions of $H^{1+s}(K)$ is given as
follows; the proof will be stated in Appendix
\ref{sec:ap}. 

\begin{lemma}
\label{la:11a}
(Trace inequality) Let $s\in (1/2,1)$. For $K\in \mathcal{T}_h$, we have
%\begin{subequations}
%\label{eq:tr0}
%\begin{align}
 \begin{equation}
	  \|\nabla v\|_{L^2(e)}^2 \le
    C_{1+s,\mathrm{T}}h_e^{-1}\left(|v|_{H^1(K)}^2+h_K^{2s}|v|_{H^{1+s}(K)}^2\right)\quad
    (v\in
    H^{1+s}(K)).\label{eq:tr2a}  
 \end{equation}
%   \end{align}
%\end{subequations}
\end{lemma}

Moreover, we deduce the
following lemma in exactly the same way as the proof of Lemma
\ref{la:7} using \eqref{eq:tr2a} instead of \eqref{eq:tr2}.  

\begin{lemma}
 \label{la:7a}
 Let $s,t\in (1/2,1]$. 
 For any $\eta_{\min}>0$, there exists a positive constant
 $C_{\mathrm{b},s,t}=C_{\mathrm{b},s,t}(\alpha,\eta_{\min},d,C_{1+s,\mathrm{T}},C_{1+t,\mathrm{T}},s,t)$ such that
       \begin{equation}
	\label{eq:bdd2}
	 B_h(\bm{w},\bm{v})\le C_{\mathrm{b},s,t}\|\bm{w}\|_{1+s,h}\|\bm{v}\|_{1+t,h}\qquad
	 (\bm{w}\in \bm{V}^{1+s}(h),\ 
	 \bm{v}\in \bm{V}^{1+t}(h)).
       \end{equation}
\end{lemma}

\begin{thm}
 \label{th:1}
 Let $u\in V$ be the solution of \eqref{eq:2} and assume that
 \eqref{eq:reg} for some $s\in (1/2,1]$. Set $\bm{u}\in
 \bm{V}^{1+s}(h)$ as in Lemma \ref{la:3}. Moreover, let $\bm{u}_h=(u_h,\hat{u}_h)\in
 \bm{V}_h$ be the solution of \eqref{eq:10}. Then, we have the Galerkin
 orthogonality
 \begin{equation}
   B_h(\bm{u}-\bm{u}_h,\bm{v}_h)=0\qquad (\forall \bm{v}_h\in\bm{V}_h).
 \label{eq:28}
 \end{equation}
 Moreover,
 \begin{equation}
  \label{eq:29}
   \|\bm{u}-\bm{u}_h\|_{1+s,h}\le C\inf_{\bm{v}_h\in\bm{V}_h}\|\bm{u}-\bm{v}_h\|_{1+s,h}.
 \end{equation}
% with $C=1+C_{\mathrm{b},s,1}/C_{\mathrm{c}}$. 
\end{thm}

\begin{proof}
 Let $\bm{v}_h\in V_h$ be arbitrarily. Then, \eqref{eq:28} is a
 consequence of \eqref{eq:10} and Lemma \ref{la:3}. On the other hand, 
\begin{align*}
C_{\mathrm{c}} \|\bm{v}_h-\bm{u}_h\|_{2,h}^2
 &\le B_h(\bm{v}_h-\bm{u}_h,\bm{v}_h-\bm{u}_h) & (\mbox{by }\eqref{eq:coer})\\
&\le B_h(\bm{v}_h-\bm{u},\bm{v}_h-\bm{u}_h)+B_h(\bm{u}-\bm{u}_h,\bm{v}_h-\bm{u}_h) & \\
&\le
 B_h(\bm{v}_h-\bm{u},\bm{v}_h-\bm{u}_h)
 &(\mbox{by \eqref{eq:28}}) \\
&\le
 C_{\mathrm{b},s,1}\|\bm{v}_h-\bm{u}\|_{1+s,h}\|\bm{v}_h-\bm{u}_h\|_{2,h}
 &(\mbox{by \eqref{eq:bdd2}}) 
\end{align*}
This implies
 \[
 \|\bm{v}_h-\bm{u}_h\|_{2,h}
 \le
 \frac{C_{\mathrm{b},s,1}}{C_{\mathrm{c}}}\|\bm{v}_h-\bm{u}\|_{1+s,h}.
 \]
We apply the triangle inequality to obtain
\begin{align*}
 \|\bm{u}-\bm{u}_h\|_{1+s,h}
 &\le 
 \|\bm{u}-\bm{v}_h\|_{1+s,h}+
 C\|\bm{v}_h-\bm{u}_h\|_{2,h}\\
 &\le 
 \|\bm{u}-\bm{v}_h\|_{1+s,h}+
 C\|\bm{v}_h-\bm{u}\|_{1+s,h},
\end{align*}
which gives \eqref{eq:29}.  
\end{proof}

\begin{thm}
 \label{th:2}
 Under the same assumptions of Theorem \ref{th:1}, we have  
  \begin{equation}
   \|\bm{u}-\bm{u}_h\|_{1+s,h}\le
Ch^s\left(\|u\|_{H^{1+s}(\Omega_1)}+\|u\|_{H^{1+s}(\Omega_2)}\right).
    \label{eq:40}
  \end{equation}
% where $C_{\star}$ denotes a positive constant independent of $h$. 
\end{thm}

\begin{proof}
 It is done by the standard method; see \cite[Paragraph 4.3]{abcm02} for
 example. However, we state the proof, since it is not apparent how to
 estimate the third term of the left-hand side of \eqref{eq:21b}.
 First, we introduce $u_I\in V_h$ as follows. 
 Let
 $K\in\mathcal{T}_h$ and let $u_{I,K}=(u_I)|_K\in P_{k}(K)$ be the Lagrange
 interpolation of $u|_K$. We remark here that $u_I$ is well-defined, since
 $u|_K\in H^{1+s}(K)$. Further, we introduce $\hat{u}_I\in \hat{V}_h$ by setting
 $\hat{u}_I|_e=(u_{I,K_1}|_e+u_{I,K_2}|_e)/2$ for
 $e\in\mathcal{E}_{h,0}\cup\mathcal{E}_{h,\Gamma}$, $e=\partial
 K_1\cap\partial K_2$ and $\hat{u}_I|_e=u_{I,K}|_e$ for
 $e\in\mathcal{E}_{h,\partial\Omega}$, $e\subset\partial K$. 
 Then, letting $\bm{w}_h=(u_I,\hat{u}_I)\in \bm{V}_h$, we derive an
 estimation for $\|\bm{u}-\bm{w}_h\|_{1+s,h}$.
 %Below, we use the letter
 %$C$ to denote positive constants independent of $h$. 

For $e\in\mathcal{E}_{h,0}\cup\mathcal{E}_{h,\Gamma}$, $e\subset \partial
 K$, we have by \eqref{eq:tr1}
 \begin{equation*}
  %\label{eq:47a}
  {\frac{\eta_e}{h_e}}
 \left\|u-u_I\right\|_{L^2(e)}^2
\le Ch_e^{-2}\left(\|u-u_I\|_{L^2(K)}^2+h_K^{2}|u-u_I|_{H^1(K)}^2\right)
 \end{equation*}
 Hence, using (H3), 
  \begin{equation*}
%\label{eq:47a}
 \sum_{K\in\mathcal{T}_h}\sum_{e\subset\partial K}
 {\frac{\eta_e}{h_e}}
 \left\|u-u_I\right\|_{L^2(e)}^2
 \le C\sum_{K\in\mathcal{T}_h}
  \left(h_K^{-2}\|u-u_I\|_{L^2(K)}^2+|u-u_I|_{H^1(K)}^2\right).
 \end{equation*}
On the other hand, for
 $e\in\mathcal{E}_{h,0}\cup\mathcal{E}_{h,\Gamma}$, $e=\partial
 K_1\cap\partial K_2$,
\[
 \|\hat{u}-\hat{u}_I\|_{L^2(e)}^2
\le  C
 \left(\|u|_{K_1}-{u}_{I,K_1}\|_{L^2(e)}^2+\|u|_{K_2}-{u}_{I,K_2}\|_{L^2(e)}^2\right)
\]
Therefore, as above, we have  
  \begin{equation*}
%\label{eq:47a}
   \sum_{K\in\mathcal{T}_h}\sum_{e\subset\partial K}
   {\frac{\eta_e}{h_e}}
 \left\|\hat{u}-\hat{u}_I\right\|_{L^2(e)}^2
 \le C\sum_{K\in\mathcal{T}_h}
  \left(h_K^{-2}\|u-u_I\|_{L^2(K)}^2+|u-u_I|_{H^1(K)}^2\right).
 \end{equation*}
Consequently, we obtain
 \begin{multline*}
\|\bm{u}-\bm{w}_h\|_{1+s,h}^2  \\ \le C\sum_{K\in\mathcal{T}_h}
  \left(h_K^{-2}\|u-u_I\|_{L^2(K)}^2+|u-u_I|_{H^1(K)}^2+h_K^{2s}|u-u_I|_{H^{1+s}(K)}^2\right).  
 \end{multline*}

At this stage, we recall
 \[
  |u-u_I|_{H^{t}(K)}\le Ch_K^{s+1-t}|u|_{H^{1+s}(K)}\quad (0\le t\le
 2),
 \]
 where $|\cdot|_{H^0(K)}$ is understood as $\|\cdot\|_{L^2(K)}$. See,
 for example, \cite[Theorems 2.19, 2.22]{fei89} where the case of
 integer $t$ is explicitly mentioned. However, the extension to the case of
 non-integer $t\in [0,1+s]$ is straightforward, since the imbedding
 $H^{t}(K)\subset H^{1+s}(K)$ is continuous. Combining those inequalities, we deduce 
 \[
  \|\bm{u}-\bm{w}_h\|_{1+s,h}\le Ch^s\left(\|u\|_{H^{1+s}(\Omega_1)}+\|u\|_{H^{1+s}(\Omega_2)}\right),
 \]
 which completes the proof. 
 % That is, we have
% \begin{equation}
%\label{eq:46}
%   \sum_{K\in\mathcal{T}_h}|u-u_I|_{H^1(K)}^2+
%   \sum_{K\in\mathcal{T}_h}h_K^{2s}|u-u_I|_{H^{1+s}(K)}^2\le Ch^{2s}N_s(u)^2,
% \end{equation}
% where $N_s(u)=\|u\|_{H^{1+s}(\Omega_1)}+\|u\|_{H^{1+s}(\Omega_2)}$. 
% We consider the case 
% Set $\bm{u}\in
% \bm{V}^{1+s}(h)$ as in Lemma \ref{la:3}. 
\end{proof}

\begin{thm}
 \label{th:3}
 Under the same assumptions of Theorem \ref{th:1}, we have  
  \begin{equation*}
   \|u-u_h\|_{L^2(\Omega)}\le
Ch^{2s}\left(\|u\|_{H^{1+s}(\Omega_1)}+\|u\|_{H^{1+s}(\Omega_2)}\right).
    %\label{eq:40a}
  \end{equation*}
\end{thm}

 \begin{proof}
  We follow the Aubin--Nitsche duality argument. 
 Set $\bm{e}_h=\bm{u}-\bm{u}_h\in \bm{V}_h$ with ${e}_h=u-u_h\in V_h$, $\hat{e}_h=\hat{u}-\hat{u}_h\in \hat{V}_h$ and consider the adjoint problem: Find $\psi\in
 V$ such that 
\begin{equation} 
\label{eq:71}
 \gamma_1\psi_1 - \gamma_2\psi_2=0\mbox{ on }\Gamma, \qquad
 a(v,\psi)=\int_\Omega ve_h~dx\quad (\forall v\in V).
\end{equation}
 (Note that we have taken $f=e_h$, $g_D=0$, $g_N=0$ and used the symmetry of $a$.)
In view of \eqref{eq:reg}, we have $\psi_1\in H^{1+s}(\Omega_1)$, $\psi_2\in H^{1+s}(\Omega_2)$ and
   \begin{equation}
    \label{eq:72}
%     \|\psi\|_{H^{1+s}(\Omega_1)}+
     %     \|\psi\|_{H^{1+s}(\Omega_2)}
N_s(\psi)\le C\|e_h\|_{L^2(\Omega)}.
   \end{equation}
  As is verified in Lemma \ref{la:2}, $\bm{\psi}=(\psi,\hat{\psi})\in
  \bm{V}^{1+s}(h)$ satisfies
\begin{equation*}
%\label{eq:73}
%\tag{\ref{eq:10}a}
 B_h(\bm{v},\bm{\psi})=\int_\Omega ve_h~dx\quad (\forall \bm{v}\in\bm{V}(h)).
\end{equation*}
HDG scheme for \eqref{eq:71} reads as follows: Find $\bm{\psi}_h\in\bm{V}_h$ such that 
  \begin{equation*}
%\label{eq:73}
%\tag{\ref{eq:10}a}
 B_h(\bm{v}_h,\bm{\psi}_h)=\int_\Omega ve_h~dx\quad (\forall \bm{v}_h\in\bm{V}_h).
\end{equation*}
Then, we have 
  \begin{align*}
   \|e_h\|_{L^2(\Omega)}^2
   &=B_h(\bm{e}_h,\bm{\psi}) =B_h(\bm{e}_h,\bm{\psi}-\bm{\psi}_h) &(\mbox{by \eqref{eq:28}})\\
   &\le C\|\bm{e}_h\|_{1+s,h}\|\bm{\psi}-\bm{\psi}_h\|_{1+s,h} &(\mbox{by \eqref{eq:bdd2}})\\
   &\le Ch^sN_s(u)\cdot h^sN_s(\psi) &(\mbox{by \eqref{eq:40}})\\
   &\le Ch^{2s}N_s(u)\cdot \|e_h\|_{L^2(\Omega)}, &(\mbox{by \eqref{eq:72}})
  \end{align*}
 which completes the proof. 
 \end{proof}

 \section{Numerical examples}
 \label{sec:9}

 In this section, we confirm the validity of error estimates described
 in Theorems \ref{th:2} and \ref{th:3} using simple
 numerical examples.
 
\begin{ex}
 \label{ex:1}
Set $\Omega_1=(0,1)\times (0,1/2)$, $\Omega_2=(0,1)\times (1/2,1)$ and 
 consider
 \begin{subequations}
  \label{eq:35}
\begin{gather}
 A=\lambda I,\qquad
 \lambda=
  \begin{cases}
   4 & \mbox{ in }\Omega_1\\
   1 & \mbox{ in }\Omega_2,
  \end{cases}
 \quad (I\mbox{: the identity matrix}),\\
 f= \begin{cases}
   8\pi^2 \sin(\pi x_1)\sin(\pi x_2) & \mbox{ in }\Omega_1\\
   -2\pi^2 \sin(\pi x_1)\sin(\pi x_2) & \mbox{ in }\Omega_2.
  \end{cases}
\end{gather}
 \end{subequations}
 The exact solution is given as 
 \begin{equation*}
u= \begin{cases}
   \sin(\pi x_1)\sin(\pi x_2) & \mbox{ in }\Omega_1\\
   -\sin(\pi x_1)\sin(\pi x_2) & \mbox{ in }\Omega_2.
  \end{cases}
 \end{equation*}
 (Functions $g_D$ and $g_N$ are computed by $u$.)
 It is apparent that $u_1\in H^2(\Omega_1)$ and $u_2\in H^2(\Omega_2)$
 so that we are able to apply Theorems \ref{th:2} and \ref{th:3} for
 $s=1$.
 \end{ex}
 
\begin{figure}[htb]
%     \begin{minipage}{.49\textwidth}
	 \begin{center}
	  \includegraphics[width=.5\textwidth]{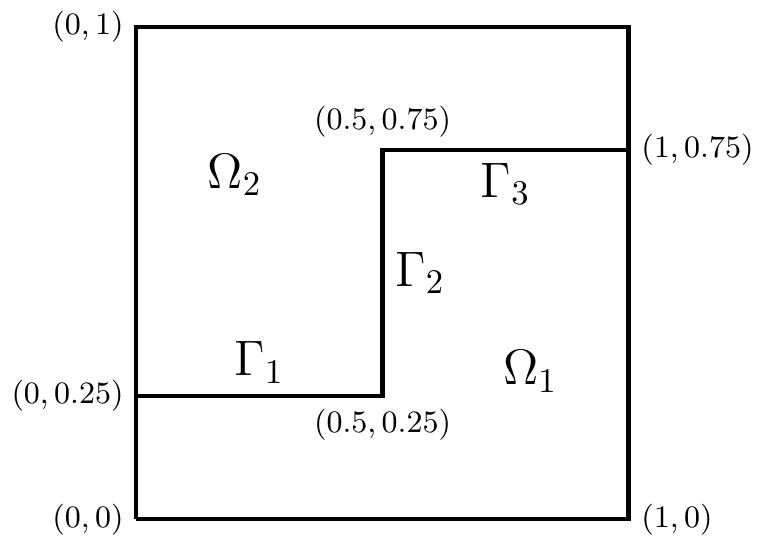}
%	  (a) 
	 \end{center}
 %   \end{minipage}
  % \begin{minipage}{.49\textwidth}
%	\begin{center}
%	 \includegraphics[width=.8\textwidth]{rate11-crop.pdf}\\
%	 (b) 
%	\end{center}
 %  \end{minipage}
   \caption{Examples of $\Omega_1$, $\Omega_2$ and $\Gamma$.}
\label{f:11}
\end{figure}

\begin{ex}
 \label{ex:2}
$\Omega$, $\Omega_1$ and $\Omega_2$ are given as shown for
 illustration in Fig.~\ref{f:11}. $\Gamma$ is set as
 $\Gamma=\Gamma_1\cup\Gamma_2\cup\Gamma_3$. We use $A$ and $f$
 defined as \eqref{eq:35}. Functions $g_D$ and $g_N$ are given as
 \[
 g_D=
  \begin{cases}
2\sin(\pi x_1) & \mbox{ on }\Gamma_1\\
2& \mbox{ on }\Gamma_2\\
2\sin(\pi x_1) & \mbox{ on }\Gamma_3,
  \end{cases}
 \qquad 
  g_N=0\mbox{ on }\Gamma.
 \]
In this case, we have $u_1\in H^{1+s}(\Omega_1)$ and $u_2\in
 H^{1+s}(\Omega_2)$ for some $s\in (1/2,1)$, since $\Omega_1$ and
 $\Omega_2$ have concave corners.   
\end{ex}

 We use the $Q_1$ element for $V_h$ on uniform rectangular meshes and
 $P_1$ for $\hat{V}_h$ (see Remark \ref{rem:17}). Set
 \begin{equation}
E_h=|u-u_h|_{H^1(\mathcal{T}_h)},\qquad 
e_h=\|u-u_h\|_{L^2(\Omega)}. 
  \label{eq:33}
 \end{equation}
For Example \ref{ex:2}, we use numerical solutions $u_{h'}$ with extra
fine mesh $h'$ instead of the exact solution $u$. 
 We examine $E_h$ and $e_h$ together with
\[
  R_h=\frac{\log E_h-\log E_{h/2}}{\log 2},\qquad 
  r_h=\frac{\log e_h-\log e_{h/2}}{\log 2}
\]
for several $h$'s.

Results are reported in Tab. \ref{tab:1} and \ref{tab:11} for Examples
\ref{ex:1} and \ref{ex:2}, respectively.  
We observe from those tables theoretical convergences with $s=1$ and
$s\in (1/2,1)$, respectively, for Examples \ref{ex:1} and \ref{ex:2} actually take place.
%On the other hand, 
%In Fig.~\ref{f:10}, we present $(\log h, \log
% |u-u_h|_{H^1(\mathcal{T}_h)})$ and 
%  $(\log h, \log \|u-u_h\|_{L^2(\Omega)})$. Results show that

\begin{table}
\begin{center}
\begin{tabular}{c|cc|cc}
\hline
$h$ & $E_h$ & $R_h$ & $e_h$ & $r_h$ \\
\hline
\hline
0.06250 & 1.62$\cdot 10^{-1}$ &      &4.14$\cdot 10^{-3}$ &\\
0.03125 & 8.13$\cdot 10^{-2}$ &1.00 &1.04$\cdot 10^{-3}$ &1.99\\
0.01563 & 4.06$\cdot 10^{-2}$ &1.00  &2.60$\cdot 10^{-4}$ &2.00\\
0.00781 & 2.04$\cdot 10^{-2}$ &0.99 &6.50$\cdot 10^{-5}$ &2.00\\
0.00391 & 1.02$\cdot 10^{-2}$ &1.00  &1.63$\cdot 10^{-5}$ &2.00\\
0.00195 & 5.08$\cdot 10^{-3}$ &1.01  &4.09$\cdot 10^{-6}$ &1.99\\
 \hline
\end{tabular}
 \end{center}
\caption{Errors and convergence rates for Example \ref{ex:1}.}
\label{tab:1}
\end{table}

%\begin{figure}[htb]
% 	 \begin{center}
%	  % rate1
%	 \includegraphics[width=.5\textwidth]{rate1-crop.pdf}
%	\end{center}
%   \caption{$\log h$ versus $\log\mbox{~ERR}$, where $\mbox{ERR}=|u-u_h|_{H^1(\mathcal{T}_h)},\|u-u_h\|_{L^2(\Omega)}$.}
%\label{f:10}
%\end{figure}

\begin{table}
 \begin{center}
\begin{tabular}{c|cc|cc}
\hline
$h$ & $E_h$ & $R_h$& $e_h$ & $r_h$ \\
\hline
\hline
0.06250 & 1.49$\cdot 10^{-1}$&	     &  1.37$\cdot 10^{-3}$&	\\
0.03125 & 7.60$\cdot 10^{-2}$&	0.98&	3.48$\cdot 10^{-4}$&	1.98\\
0.01563 & 3.88$\cdot 10^{-2}$&	0.97&	8.83$\cdot 10^{-5}$&	1.98\\
0.007813& 1.99$\cdot 10^{-2}$&	0.97&	2.26$\cdot 10^{-5}$&	1.97\\
0.003906& 1.02$\cdot 10^{-2}$&	0.96&	5.86$\cdot 10^{-6}$&	1.95\\
 \hline
\end{tabular}
 \end{center}
\caption{Errors and convergence rates for Example \ref{ex:2}.}
\label{tab:11}
\end{table}

\section*{Acknowledgement}
NS is supported by %CREST, Japan Science and Technology Agency and
JSPS KAKENHI Grant Number 15H03635, 15K13454.

%%%%%%%%%%%%%%%%%%%%%%%%%%%%%%%%%%%%%%%%%%%%%%%%%%%%%%%%%%%%%%%%%%%%
%% references 
%%%%%%%%%%%%%%%%%%%%%%%%%%%%%%%%%%%%%%%%%%%%%%%%%%%%%%%%%%%%%%%%%%%
%%% bibtex ms17-1c
\bibliographystyle{plain}
\bibliography{ms17-1} 
%%%%%%%%%%%%%%%%%%%%%%%%%%%%%%%%%%%%%%%%%%%%%%%%%%%%%%%%%%%%%%%%%%%
%%%%%%%%%%%%%%%%%%%%%%%%%%%%%%%%%%%%%%%%%%%%%%%%%%%%%%%%%%%%%%%%%%%

%%%%%%%%%%%%%%%%%%%%%%%%%%%%%%%%%%%%%%%%%%%%%%%%%%%%%%%%%%%%%%%%%%%
\appendix 
\section{Proof of Lemma \ref{la:11a}}
\label{sec:ap}

Let $s\in (1/2,1)$. 
Let $K\in \mathcal{T}_h$ and $e\subset \partial K$.

The fractional order Sobolev space $H^s(K)$ is defined as  
%\begin{subequations}
% \label{eq:33}
\begin{equation*}
% \label{eq:33c}
  H^{s}(K)=\{v\in L^2(K)\mid \|v\|_{H^{s}(K)}^2=\|v\|_{L^2(K)}^2+|v|_{H^{s}(K)}^2<\infty\},
\end{equation*}
where 
\begin{equation*}
|v|_{H^s(K)}^2=\int \hspace{-2mm}
 \int_{K\times K}\frac{~|v(x)-v(y)|^2}{|x-y|^{d+2s}}~dxdy. 
 %\label{eq:33a}
\end{equation*}
%\end{subequations}

It suffices to prove
\begin{equation}
	  \|v\|_{L^2(e)}^2 \le
	  C_{s,\mathrm{T}}h_e^{-1}\left(\|v\|_{L^2(K)}^2+h_K^{2s}|v|_{H^s(K)}^2\right)\qquad
	  (v\in
 H^s(K)), \label{eq:tr1a} 
\end{equation}
since the desired inequality \eqref{eq:tr2a} is a direct consequence of \eqref{eq:tr1a}.

Suppose that $\tilde{K}$ is the reference element in $\mathbb{R}^d$ with
$\operatorname{diam}(\tilde{K})=1$. Moreover, let $\tilde{e}\subset\partial\tilde{K}$ be
a face ($d=3$)/edge ($d=2$) of $\tilde{K}$. Trace theorem implies 
\begin{equation*}
% \label{eq:tr12}
\|\tilde{v}\|_{L^2(e)}^2\le
\tilde{C}\left(\|\tilde{v}\|_{L^2(\tilde{K})}^2+|\tilde{v}|_{H^{s}(\tilde{K})}^2\right)\qquad (\tilde{v}\in H^1(\tilde{K})),
\end{equation*}
where $\tilde{C}$ denotes an
absolute positive constant. See \cite[Theorem
1, \S V.1.1]{jw84} for example.
%Hereainafther, we use the letter $C$ to denote positive
%constants depending only on $s$. 

Suppose that $\Phi(\xi)=B\xi+c$, $B\in\mathbb{R}^{d\times
d}$, $c\in\mathbb{R}^d$, is the affine mapping which maps $\tilde{K}$ onto
$K$; $K=\Phi(\tilde{K})$. We know
\begin{equation*}
% \label{eq:ap10}
  \|B\|=\sup_{|\xi|=1}|B\xi|\le\frac{h_K}{\tilde{\rho}},\quad 
  \|B^{-1}\|\le \frac{\tilde{h}}{\rho_K},\quad
  d\xi=\frac{\operatorname{meas}_{d}(\tilde{K})}{\operatorname{meas}_{d}({K})} dx,
\end{equation*}
where $\tilde{h}={h}_{\tilde{K}}$, $\tilde{\rho}={\rho}_{\tilde{K}}$ and
$\operatorname{meas}_d(K)$ denotes the $\mathbb{R}^d$-Lebesgue measure
of $K$. Moreover,
\begin{equation*}
% \label{eq:ap11}
\frac{|x|}{|B^{-1}x|}\le \sup_{\xi\in\mathbb{R}^d}
\frac{|B\xi|}{|\xi|}=\|B\|\qquad (x\in\mathbb{R}^d,x\ne 0).
\end{equation*}
We recall that there exists a positive constant $\nu_2$ that independent
of $h$ such that $h_K/\rho_K\le \nu_2$ ($\forall
K\in\forall\mathcal{T}_h\in \{\mathcal{T}_h\}_h$) by the
shape-regularity of the family of triangulations. 

%We consider the restriction $\Phi_e$ of $\Phi$ into $e$.
Now we can state the proof of \eqref{eq:tr1a}.  
By the density, it suffices to consider \eqref{eq:tr1a} for $v\in C^1(K)$. Set $\tilde{v}=v\circ\Phi\in
C^1(\tilde{K})$. Then,
\begin{equation*}
%\label{eq:ap13}
 \int_{\tilde{K}}\tilde{v}^2d\xi
 =
 \frac{\operatorname{meas}_{d}(\tilde{K})}{\operatorname{meas}_{d}({K})}
 \int_{{K}}{v}^2dx
 \le C \rho_K^{-d}\|v\|_{L^2(K)}^2%\int_{{K}}{v}^2dx
\end{equation*}
and 
\begin{align*}
\int\hspace{-2mm}\int_{\tilde{K}\times\tilde{K}}\frac{~|\tilde{v}(\xi)-\tilde{v}(\eta)|^2}{|\xi-\eta|^{d+2s}}~d\xi
 d\eta
 &\le
 \left(\frac{\operatorname{meas}_{d}(\tilde{K})}{\operatorname{meas}_{d}({K})}\right)^2 
\int\hspace{-2mm}\int_{{K}\times{K}}\frac{~|v(x)-v(y)|^2}{~|B^{-1}x-B^{-1}y|^{d+2s}}~dxdy \\
&\le
 C\rho_K^{-2d} 
\cdot \|B\|^{d+2s}\int\hspace{-2mm}\int_{{K}\times{K}}\frac{~|v(x)-v(y)|^2}{~|x-y|^{d+2s}}~dxdy\\
&\le
 Ch_K^{2s}\nu_2^d\rho_K^{-d} %\left(\frac{h_K}{\rho_K}\right)^d
 \int\hspace{-2mm}\int_{{K}\times{K}}\frac{~|v(x)-v(y)|^2}{~|x-y|^{d+2s}}~dxdy.
\end{align*}
Using those inequalities, we have 
\begin{align*}
 \|\tilde{v}\|_{L^2(e)}^2
 &=
 \frac{\operatorname{meas}_{d-1}(e)}{\operatorname{meas}_{d-1}(\tilde{e})}
 \int_{\tilde{e}}\tilde{v}(\xi)^2~d\xi\\
 &\le Ch_e^{d-1}\cdot 
 \tilde{C}\left(\int_{\tilde{K}}\tilde{v}^2d\xi+
\int\hspace{-2mm}\int_{\tilde{K}\times\tilde{K}}\frac{~|v(\xi)-v(\eta)|^2}{|\xi-\eta|^{d+2s}}~d\xi
 d\eta\right).\\
 &\le C\nu_1^dh_e^{-1} %\left(\frac{h_e}{\rho_K}\right)^{d}
 \left(\|v\|_{L^2(K)}^2+h_K^{2s}|v|_{H^s(K)}^2 \right),
\end{align*}
which completes the proof. 
%\begin{remark}
%\label{rem:trace}
%Trace inequalily \eqref{eq:tr12} is well-known when $\tilde{K}$ is a sufficiently% smooth
% boubded domain; see for example \cite{lm68}. We recall \eqref{eq:tr12} for a
% bounded Lipschitz domain $\tilde{K}$ in $\mathbb{R}^d$. Let
% $1/2<s<3/2$. Then, the space $$ and let $\gamma$ be a part of the boundary $\partial\omega$
% of $\omega$ with $\operatorname{meas}_{d-1}(\gamma)>0$.
% Then, there exists a linear continuous operator $\operatorname{Tr}$ of
% $H^{s}(\omega)\to H^{s-1/2}(\gamma)$, the trace $\psi$ of $w$ into
% $\gamma$ is well-defined 
% there exists a positive constant
% $\tilde{C}_s$ depending only on $\tilde{\Omega}$ and $s$ such that
% \[
% \|\tilde{v}\|_{H^{s-1/2}(\tilde{e})}\le \tilde{C}_s
% \left(\|\tilde{v}\|_{L^2(\tilde{K})}^2+|\tilde{v}|_{H^{s}(\tilde{K})}^2\right)
% \]
%  
%\end{remark}

%%%
%%%
\end{document}